%% file: eooinf.tex
\documentclass[10pt,letterpaper,twoside]{article}

\usepackage[letterpaper,left=1in,right=1in,top=1.5in,bottom=1.5in]{geometry}

\title{Filtrations and cohomology III: cohomology of $\bE_\infty$-rings}\author{Benjamin Antieau}\date{\today}

\input{preamble}

\begin{document}

\maketitle
\begin{abstract}
    \noindent
    We discuss filtrations arising from de Rham-type cohomology theories for $\bE_\infty$-rings and
    $\bE_n$-rings. Examples include the HKR filtration on relative topological Hochschild homology, the Hodge
    filtration on $\bE_\infty$-infinitesimal cohomology, and the Hodge filtration on $\bE_\infty$-de Rham cohomology. 
\end{abstract}

\renewcommand{\baselinestretch}{0.75}\small
\tableofcontents
\renewcommand{\baselinestretch}{1.0}\normalsize

\section{Introduction}

We give constructions, via universal properties, of certain multiplicative filtrations associated to
maps of $\bE_\infty$-ring spectra (or $\bE_n$-ring spectra) and explain how these are related to spectral sequences and
filtrations previously studied in the literature. The final section gives a connection to
Goodwillie calculus. Most of the main results have proofs identical to those in~\cite{antieau_crystallization,antieau_gm}, so we omit
them. Closely related ideas were developed several years ago by Tyler Lawson
and explained by him at a conference in Boulder; see the video~\cite{lawson-video}.

\paragraph{Notation.}
We adopt all of the notation and terminology from the
prequels~\cite{antieau_crystallization, antieau_gm}.
An undecorated module category such as $\Mod_k$ always denotes the $\infty$-category of left
$k$-modules, often written $\LMod_k$.
If $k$ is an $\bE_\infty$-ring, let $\CAlg_k$ be the $\infty$-category of $\bE_\infty$-$k$-algebras
and let $\FCAlg_k$ be the $\infty$-category of filtered $\bE_\infty$-$k$-algebras.
An object of $\FCAlg_k$ is sometimes called a decreasing $\bE_\infty$-filtration or a decreasing
multiplicative filtration.
Let $\widehat{\FCAlg}_k$ be the $\infty$-category of complete filtered $\bE_\infty$-$k$-algebras.
If $k=\bS$ is the sphere spectrum, we simply write these as $\CAlg$, $\FCAlg$, and
$\widehat{\FCAlg}$. Often, we view an $\bE_\infty$-ring $R$ as a filtered $\bE_\infty$-ring
without comment; in these cases, we implicitly mean the filtered $\bE_\infty$-ring $\ins^0R$
corresponding to the filtration $$\cdots\rightarrow 0\rightarrow R=R=R=\cdots,$$ where the first non-zero term
is in filtration weight zero.

\paragraph{Acknowledgments.}
We thank Bhargav Bhatt, Sanath Devalapurkar, Bjørn Dundas, Elden Elmanto, John Francis, Paul Goerss, Mike Hill,
Achim Krause, Tyler Lawson, Deven Manam, Akhil Mathew, Thomas Nikolaus, Arpon Raksit, Birgit Richter, and Nick Rozenblyum for helpful conversations.
Special thanks to Bjørn Dundas, Tyler Lawson, and Birgit Richter for exceptionally helpful comments
on the first draft.

This project was supported by NSF grant DMS-2152235 and the Simons Collaboration on Perfection.
We also thank AIM for its hospitality: we learned of the work of Hill--Lawson discussed in Example~\ref{ex:hill_lawson} from Michael
Hill there in October 2025.

\section{Infinitesimal cohomology}

Our basic notion is the $\bE_\infty$-analogue of infinitesimal cohomology introduced
in~\cite{antieau_crystallization}.
The definition is motivated by Raksit's definition of derived de Rham cohomology
in~\cite{raksit}.

\begin{construction}[$\bE_\infty$-infinitesimal cohomology]
    Let $k$ be an $\bE_\infty$-ring and let $\FCAlg_k$ be the $\infty$-category of filtered
    $\bE_\infty$-algebras over $k$. The functor $\gr^0\colon\FCAlg_k\rightarrow\CAlg_k$ preserves
    limits and colimits, so is the right adjoint of an adjunction
    $$\F^\star_\H\InfEoo_{-/k}\colon\CAlg_k\rightleftarrows\FCAlg_k\colon\gr^0,$$
    where $\F^\star_\H\InfEoo_{-/k}$ is the left adjoint, by definition.
    Given $R\in\CAlg_k$, we call $\InfEoo_{R/k}$ the $\bE_\infty$-infinitesimal cohomology of $R$
    relative to $k$. By definition, it comes equipped with a filtration $\F^\star_\H\InfEoo_{R/k}$ called the 
    Hodge filtration. We let $\InfhatEoo_{R/k}$ be the completion of $\InfEoo_{R/k}$ with respect
    to the Hodge filtration and let $\F^\star_\H\InfhatEoo_{R/k}$ denote the induced Hodge
    filtration on the completion.
    There is an adjunction
    $$\F^\star_\H\InfhatEoo_{-/k}\colon\CAlg_k\rightleftarrows\widehat{\FCAlg}_k\colon\gr^0$$
    giving the universal property of the completed form of $\bE_\infty$-infinitesimal cohomology.
\end{construction}

\begin{lemma}\label{lem:graded}
    There is a natural equivalence
    $$\Sym^*_R(\L_{R/k}^{\bE_\infty}[-1])\rightarrow\gr^*_\H\InfEoo_{R/k}$$
    of graded $\bE_\infty$-$R$-algebras.
\end{lemma}

\begin{proof}
    See~F\&C.I.8.13 or~\cite[Thm.~5.3.6]{raksit}.
\end{proof}

\begin{construction}[Hodge spectral sequence]
    Given a filtered spectrum $\F^\star M$ the boundary maps
    $\delta\colon\gr^iM\rightarrow\gr^{i+1}M[1]$ assemble
    into a kind of coherent cochain complex
    $$\cdots\rightarrow\gr^{-1}M[-1]\rightarrow\gr^0M\rightarrow\gr^1M[1]\rightarrow\cdots.$$ For
    example, an easy exercise implies that $\delta^2\we 0$. In
    fact, as shown by Ariotta~\cite{ariotta}, complete filtered spectra are equivalent to coherent
    cochain complexes in spectra.\footnote{This equivalence can also be seen as a form of Koszul
    duality and as such appears in~\cite[Sec.~3.2]{raksit}.} Under the equivalence between complete filtrations and coherent cochain complexes,
    $\F^\star_\H\InfhatEoo_{R/k}$ corresponds to a coherent cochain complex of the form
    $$R\rightarrow\L_{R/k}^{\bE_\infty}\rightarrow\Sym^2_R(\L_{R/k}^{\bE_\infty}[-1])[2]\rightarrow\Sym^3_R(\L_{R/k}^{\bE_\infty}[-1])[3]\rightarrow\cdots.$$
    The associated Hodge spectral sequence takes the form
    $$\E_1^{s,t}=\H^{s+t}(\Sym^s_R(\L_{R/k}^{\bE_\infty}[-1]))\Rightarrow\H^{s+t}(\InfhatEoo_{R/k}),$$
    where for any spectrum $M$ we set $\H^i(M)=\pi_{-i}(M)$.
\end{construction}

The following result is easy, but important.

\begin{theorem}[Poincar\'e lemma]\label{lem:intro_poincare}
    For any map $k\rightarrow R$ of $\bE_\infty$-rings, the natural map $k\rightarrow\InfEoo_{R/k}$
    is an equivalence.
\end{theorem}

\begin{proof}
    See F\&C.I.9.6.
    Taking the colimit gives a left adjoint functor
    $\colim\colon\FCAlg_k\rightarrow\CAlg_k$. It admits a right adjoint given by the constant filtration
    functor, which takes an $\bE_\infty$-algebra $R$ over $k$ to the constant $\bZ^\op$-indexed
    filtration $$\cdots =R=R=R=\cdots.$$
    The composition of the right adjoints is the constant functor $\CAlg_k\leftarrow\CAlg_k$ on the
    final object of $\CAlg_k$, the zero ring. It follows that the composition of the left adjoints
    is the constant functor on the initial object, i.e., on $k$.
\end{proof}

\begin{definition}
    Let $\Pairs(\CAlg)$ be the $\infty$-category $\CAlg^{\Delta^1}$ consisting of maps
    $k\rightarrow R$ of $\bE_\infty$-rings. This is a symmetric monoidal $\infty$-category with the
    coCartesian symmetric monoidal structure.
\end{definition}

\begin{proposition}[Base change properties]\label{prop:base_change}
    The functors \begin{align*}
        \F^\star_\H\InfEoo_{-/-}\colon\Pairs(\CAlg)&\rightarrow\FCAlg,\\
        \F^\star_\H\InfhatEoo_{-/-}\colon\Pairs(\CAlg)&\rightarrow\widehat{\FCAlg}
    \end{align*}
    admit symmetric monoidal structures.
\end{proposition}

\begin{proof}
    See~F\&C.II.2.6.
\end{proof}

\begin{example}
    If $k\rightarrow R\rightarrow S$ are maps of $\bE_\infty$-rings, then there are natural
    equivalences
    $$\F^\star_\H\InfEoo_{S/k}\otimes_{\F^\star_\H\InfEoo_{R/k}}R\we\F^\star_\H\InfEoo_{S/R}\quad\text{and}\quad
    \F^\star_\H\InfhatEoo_{S/k}\widehat{\otimes}_{\F^\star_\H\InfhatEoo_{R/k}}R\we\F^\star_\H\InfhatEoo_{S/R}.$$
\end{example}

\begin{remark}
    As a filtered $\F^\star_\H\InfEoo_{R/k}$-algebra, $\F^\star_\H\InfEoo_{S/k}$ admits a universal
    property as described by an adjunction
    $$\F^\star_\H\InfEoo_{-/k}\colon\CAlg_R\rightleftarrows\FCAlg_{\F^\star_\H\Inf_{R/k}^{\bE_\infty}}\colon\gr^0,$$ and
    similarly in the complete setting. See~F\&C.II.2.1.
\end{remark}

Thanks to the Poincar\'e lemma, the interesting questions are as follows. (a) When is the Hodge filtration on
$\InfEoo_{R/k}$ complete? (b) What is the completion $\InfhatEoo_{R/k}$? When $\InfhatEoo_{R/k}$ is of
interest, the Hodge spectral sequence gives a new way of understanding it.

To analyze completeness and completion, we require a result on descent for canonical covers which will appear in forthcoming
work with Ryomei Iwasa and Achim Krause.

\begin{definition}[Canonical covers]\label{def:canonical} Let $f\colon (R/A)\rightarrow (S/B)$ be a map of pairs of connective $\bE_\infty$-rings,
    corresponding to a commutative diagram
    $$\xymatrix{
        A\ar[r]\ar[d]&B\ar[d]\\
        R\ar[r]&S.
    }$$
    Let $S^\bullet$ be the \v{C}ech complex of $R\rightarrow S$. Say that $f$ is a canonical cover
    if the natural map $R\rightarrow\Tot(S^\bullet)$ is an equivalence, and if this remains true
    after base change to any other connective $\bE_\infty$-ring. Note that this condition depends
    only on the map $R\rightarrow S$ and not on any other information about the morphism of pairs.
\end{definition}

\begin{remark}
    In the sense of topos theory, there are canonical covers of connective $\bE_\infty$-rings which
    are not generated by a single canonical cover in the sense of Definition~\ref{def:canonical}.
    An example of a non-finitely generated canonical cover of $\Spec\bZ$ is given
    in~\cite[\href{https://stacks.math.columbia.edu/tag/0EUE}{Tag 0EUE}]{stacks-project} in the
    context of ordinary commutative rings.
    In our work with Iwasa and Krause, we also show that Hodge-completed $\bE_\infty$-infinitesimal
    cohomology is a sheaf with respect to all canonical covers on the site opposite to the
    $\infty$-category of connective $\bE_\infty$-rings.
\end{remark}

\begin{theorem}[Descent, Antieau--Iwasa--Krause]\label{thm:descent}
    Let $f\colon(R/A)\rightarrow(S/B)$ be a canonical cover of connective $\bE_\infty$-rings. If
    $(S^\bullet/B^\bullet)$ denotes the \v{C}ech complex of $f$, then the natural map
    $$\F^\star_\H\InfhatEoo_{R/A}\rightarrow\Tot(\F^\star_\H\InfhatEoo_{S^\bullet/B^\bullet})$$ is
    an equivalence.
\end{theorem}

\begin{corollary}
    Let $(R/A)\rightarrow(S/B)$ be a canonical cover of connective $\bE_\infty$-rings with \v{C}ech
    complex $(S^\bullet/B^\bullet)$. If $M$ is a bounded below $A$-module, then the natural maps
    $$M\otimes_A\Sym^i_R(\L_{R/A}^{\bE_\infty}[-1])\rightarrow\Tot(M\otimes_A\Sym^i_{S^\bullet}(\L_{S^\bullet/B^\bullet}^{\bE_\infty}[-1]))$$
    are equivalences for all $i\geq 0$. Taking $M=A$, we find that the natural maps
    $$\Sym^i_R(\L_{R/A}^{\bE_\infty}[-1])\rightarrow\Tot(\Sym^i_{S^\bullet}(\L_{S^\bullet/B^\bullet}^{\bE_\infty}[-1]))$$
    are equivalences for all $i\geq 0$.
\end{corollary}

\begin{proof}
    See~F\&C.II.6.4.
\end{proof}

The construction of the Gauss--Manin connection in~\cite{antieau_gm} carries over to $\InfEoo$.

\begin{proposition}[Gauss--Manin filtration]
    Given maps $k\rightarrow R\rightarrow S$ of $\bE_\infty$-rings, there is a complete exhaustive
    multiplicative filtration $\F^\star_\GM\F^\star_\H\InfhatEoo_{S/k}$ on
    $\F^\star_\H\InfhatEoo_{S/k}$ with graded pieces
    $$\gr^i_\GM\F^\star_\H\InfhatEoo_{S/k}\we\F^\star_\H\InfhatEoo_{S/R}\widehat{\otimes}_R\Sym^i_R(\L_{R/k}^{\bE_\infty}[-1]).$$
\end{proposition}

\begin{proof}
    See~F\&C.II.4.12.
\end{proof}

\begin{remark}
    As a coherent cochain complex (in complete filtered spectra), the Gauss--Manin filtration takes
    the form
    $$\F^\star_\H\InfhatEoo_{S/R}\xrightarrow{\nabla}\F^{\star-1}_\H\InfhatEoo_{S/R}\widehat{\otimes}_R\L_{R/k}^{\bE_\infty}\xrightarrow{\nabla}\F^{\star-2}\InfhatEoo_{S/k}\widehat{\otimes}_R\Sym^2_R(\L_{R/k}^{\bE_\infty}[-1])[2]\rightarrow\cdots.$$
\end{remark}

\begin{definition}[Completeness]
    Let $k\rightarrow R$ be a map of $\bE_\infty$-rings with \v{C}ech complex $R^\bullet$. If $M$
    is a $k$-module, we let $M_R^\wedge=\Tot(M\otimes_kR^\bullet)$. We say that $M$ is $R$-complete if
    the natural map $M\rightarrow M_R^\wedge$ is an equivalence. This definition does not
    agree in general with other definitions of completion.
\end{definition}

We have the following three results.

\begin{proposition}\label{prop:1connective}
    If $k\rightarrow R$ is a $1$-connective map of connective $\bE_\infty$-rings, then $\F^s\InfEoo_{R/k}$ is
    $s$-connective. In particular, the Hodge filtration is complete.
\end{proposition}

\begin{proof}
    See~F\&C.II.7.1.
\end{proof}

\begin{proposition}
    Suppose that $k\rightarrow R$ is a map of connective $\bE_\infty$-rings such that
    $\pi_0k\rightarrow\pi_0S$ is surjective. Then,
    for any bounded below $k$-module, $M\widehat{\otimes}_k\F^\star\InfEoo_{R/k}$ is a complete
    filtration on $M_R^\wedge$. In particular, taking $M=k$, we find
    $k_R^\wedge\we\InfhatEoo_{R/k}$.
\end{proposition}

\begin{proof}
    See~F\&C.II.7.3.
\end{proof}

\begin{theorem}[Canonical invariance]
    Suppose that $k\rightarrow R\rightarrow S$ is a map of connective $\bE_\infty$-rings. If
    $R\rightarrow S$ is a canonical cover such that $\pi_0R\rightarrow\pi_0S$ is surjective, then
    $\Infhat_{R/k}\we\Infhat_{S/k}$.
\end{theorem}

\begin{proof}
    See~F\&C.II.7.5.
\end{proof}

\begin{example}
    The $p$-complete sphere $\bS_p$ is $\bF_p$-complete, i.e., the 
    Adams spectral sequence converges. It follows that $\bS_p\we\InfhatEoo_{\bF_p/\bS}$.
    The Hodge spectral sequence gives a spectral sequence
    converging to the stable stems which, as far as we are aware, has not appeared in the literature
    before. The $\E^1$-page receives a map from the $\E^1$-page of the Adams spectral sequence
    thanks to Example~\ref{ex:e1} and maps to the $\E^2$-page of the Adams spectral sequence.
    Specifically, Example~\ref{ex:e1} implies that the $\E^1$-page of the Adams spectral sequence
    arises as the $\E^1$-page of $\F^\star_\H\Infhat^{\bE_1}_{\bF_p/\bS}$. Its d\'ecalage
    $\Dec(\F^\star_\H\Infhat^{\bE_1}_{\bF_p/\bS})$ has associated $\E^1$-page the $\E^2$-page of
    the Adams spectral sequence (see~\cite{antieau-decalage}). But, $\Dec(\F^\star_\H\Infhat^{\bE_1}_{\bF_p/\bS})$ also admits a
    natural $\bE_\infty$-structure with $\gr^0\we\bF_p$ as it can be constructed as
    $\Tot(\tau_{\geq\star}\bF_p^\bullet)$ where $\bF_p^\bullet$ is the \v{C}ech complex of
    $\bS\rightarrow\bF_p$. It follows that there are maps
    $$\F^\star_\H\Infhat^{\bE_1}_{\bF_p/\bS}\rightarrow\F^\star_\H\InfhatEoo_{\bF_p/\bS}\rightarrow\Dec(\F^\star_\H\Infhat^{\bE_1}_{\bF_p/\bS})$$ of filtered
    $\bE_1$-algebras.
\end{example}

\section{The Quillen spectral sequence}

If $k\rightarrow R$ is a map of $\bE_\infty$-rings, then
$\L_{R/(R\otimes_kR)}^{\bE_\infty}\we\L_{R/k}^{\bE_\infty}[1]$. It follows that the Hodge
filtration on $\InfEoo_{R/(R\otimes_kR)}$ has associated graded pieces equivalent to
$\Sym^i_R(\L_{R/k}^{\bE_\infty})$. This leads to an $\bE_\infty$-version of the Quillen spectral
sequence; for the classical version, see~\cite[Thm.~6.3]{quillen-cohomology} or~\cite[\href{https://stacks.math.columbia.edu/tag/08RC}{Tag 08RC}]{stacks-project}
or~F\&C.II.8.

\begin{proposition}
    If $k\rightarrow R$ is a map of $\bE_\infty$-rings, then there is a decreasing $\bE_\infty$-filtration on $B\otimes_kB$ with associated graded pieces $\Sym^i(\L_{R/k}^{\bE_\infty})$.
    If $k$ and $R$ are connective and $\pi_0k\rightarrow\pi_0R$ is surjective, then this filtration
    is complete.
\end{proposition}

\begin{proof}
    The filtration is $\F^\star_\H\InfEoo_{R/R\otimes_kR}$.
    This is an $\bE_\infty$-algebra in $\FMod_{R\otimes_kR}$ by construction and
    $\F^0_\H\InfEoo_{R/R\otimes_kR}$ is $R\otimes_kR$ by the Poincar\'e lemma.
    If $k$ and $R$ are connective and $\pi_0k\rightarrow\pi_0R$ is surjective, then the fiber of
    $R\otimes_kR\rightarrow R$ is $1$-connective, so the filtration is complete by
    Proposition~\ref{prop:1connective}.
\end{proof}

\begin{construction}[Quillen spectral sequence]
    The Hodge spectral sequence for $\F^\star_\H\InfEoo_{R/R\otimes_kR}$ takes the form
    $$\E^1_{s,t}=\pi_{s+t}(\Sym^{-s}_R(\L_{R/k}^{\bE_\infty}))\Rightarrow\pi_{s+t}(R\otimes_kR).$$
\end{construction}

\begin{remark}
    We have that $\F^\star_\H\InfEoo_{R/(R\otimes_kR)}\we
    R\otimes_{\F^\star_\H\InfEoo_{R/k}}R$ by Proposition~\ref{prop:base_change}.
\end{remark}

\section{The HKR filtration}

Let $k\rightarrow R$ be a map of $\bE_\infty$-rings.
Recall that $\THH(R/k)\we R\otimes_{R\otimes_kR}R$. There is a canonical augmentation
$\THH(R/k)\rightarrow R$ and, with respect to this augmentation,
$\L_{R/\THH(R/k)}^{\bE_\infty}\we\L_{R/k}^{\bE_\infty}[2]$.
Following F\&C.II.9, this leads to the following form of the HKR filtration on $\THH(R/k)$.

\begin{proposition}\label{prop:hkr}
    If $k\rightarrow R$ is a map of $\bE_\infty$-rings, then there is a decreasing $\bE_\infty$-filtration
    $\F^\star_\HKR\THH(R/k)$ with $$\gr^i_\HKR\THH(R/k)\we\Sym^i_R(\L_{R/k}^{\bE_\infty}[1]).$$
    If $k$ and $R$ are connective, this filtration is complete.
\end{proposition}

\begin{proof}
    Define $\F^\star_\HKR\THH(R/k)=\F^\star_\H\Inf_{R/\THH(R/k)}$.
    The rest of the result follows from Proposition~\ref{prop:1connective}.
\end{proof}

\begin{motto}
    The HKR filtration on $\THH(R/k)$ is the Hodge filtration on $\Inf_{R/\THH(R/k)}$.
\end{motto}

\begin{remark}
    The construction leading to Proposition~\ref{prop:hkr} implies that the HKR
    filtration on $\THH(R/k)$ can be constructed by computing $\THH(R/\F^\star_\H\Inf_{R/k})$, the
    $\THH$ internal to $\FCAlg$ of $R$ relative to $\F^\star_\H\Inf_{R/k}$.
    That is, $$\F^\star_\HKR\THH(R/k)\we R\otimes_{R\otimes_{\F^\star_\H\InfEoo_{R/k}R}R}R.$$
    As a result this
    filtration is $S^1$-equivariant, but this construction does not naturally produce a filtered
    circle equivariant filtration in the sense of~\cite{mrt,raksit} or a synthetic circle action in
    the sense of~\cite{antieau-riggenbach,hedenlund-moulinos}.
\end{remark}

\begin{construction}[HKR spectral sequence]
    The HKR spectral sequence takes the form
    $$\E^1_{s,t}=\pi_{s+t}(\Sym^{-s}_R(\L_{R/k}^{\bE_\infty}[1]))\Rightarrow\pi_{s+t}\THH(R/k).$$
    It is convergent if $k$ and $R$ connective.
\end{construction}

\begin{remark}\label{rem:minasian1}
    A similar spectral sequence is constructed by Minasian in~\cite{minasian}.
    We will see that they agree in Remark~\ref{rem:minasian2}.
\end{remark}

\begin{remark}[Derived-$\bE_\infty$ comparison]
    Suppose that $k\rightarrow R$ is a map of connective derived commutative rings. Then, we can consider the
    classical HKR filtration $\HH_\fil(R/k)$, which depends on the derived structure. Here, we have
    $\gr^i\we\Lambda^i\L_{R/k}[i]$ where $\L_{R/k}$ is the cotangent complex of the map of derived
    commutative rings. We also have $\F^\star\THH(R/k)$ constructed as above. This is sensitive
    only to the $\bE_\infty$-structures. There is a map $\F^\star\THH(R/k)\rightarrow\HH_\fil(R/k)$
    and $\THH(R/k)\we\HH(R/k)$. Both filtrations are complete and hence the fiber $\F^\star M$
    computed in $\FMod_k$ is a filtration on zero.
\end{remark}

\section{Higher $\THH$ and $\bF_p$}\label{sec:fp}

Higher versions of $\THH$ have been studied going back to Loday and Pirashvili~\cite{pirashvili}.
Nowadays we know them as examples of factorization, or chiral, homology;
see~\cite{ayala-francis,ha}. Many computations
have been given in recent years, mostly stemming from B\"okstedt's calculation of $\THH(\bZ)$ or
$\THH(\bF_p)$. See for example~\cite{bhlprz,blprz,dlr,hhlrz,hklrz}. We observe that the HKR filtration extends to all of these theories.

\begin{construction}[HKR filtration on higher $\THH$]\label{const:higher_hkr}
    Let $k\rightarrow R$ be a map of $\bE_\infty$-rings.
    Let $\THH^{(n)}(R/k)$ be the result of applying the Bar construction $n$ times to the augmentation
    $R\otimes_kR\rightarrow R$. Thus, $\THH^{(0)}(R/k)\we R\otimes_kR$, $\THH^{(1)}(R/k)\we\THH(R/k)$,
    and so on. This construction is related to the more general topological chiral homology
    construction: $$\THH^{(n)}(R/k)\we\int_{S^n}^kR.$$ See~\cite{ayala-francis} or~\cite[Chap.~5]{ha} for more details.
    Here, the superscript $k$ in $\int^k$ denotes that we work internally to $\Mod_k$.

    Similarly, we can let $\F^\star_\HKR\THH^{(n)}(R/k)$ be the result of applying the Bar
    construction $n$ times to the augmentation $\InfEoo_{R/R\otimes_kR}\rightarrow R$ in $\FCAlg$.
    Thus, $\F^\star_\HKR\THH^{(0)}(R/k)$ is what we used to give our Quillen spectral sequence, while
    $\F^\star_\HKR\THH^{(1)}(R/k)$ is the $\bE_\infty$-analogue of the usual HKR spectral sequence in
    Hochschild homology. There is also a factorization homology interpretation:
    $$\F^\star_\HKR\THH^{(n)}(R/k)\we\int_{S^n}^{\F^\star_\H\InfEoo_{R/k}}R,$$
    where we work internally to $\FMod_{\F^\star_\H\InfEoo_{R/k}}$.
\end{construction}

\begin{remark}
    With these conventions, it makes sense also to let $\THH^{(-1)}(R/k)=k$ and
    $\F^\star_\HKR\THH^{(-1)}(R/k)=\F^\star_\H\InfEoo_{R/k}$.
\end{remark}

\begin{remark}
    More generally, if $k\rightarrow R$ is a map of $\bE_\infty$rings and $X$ is an anima, then we can consider the $\bE_\infty$-ring $X\otimes R$,
    the tensoring or copowering of $R$ by $X$ in $\CAlg_k$. We obtain an HKR-like filtration on
    $X\otimes R$ by instead considering $X\otimes R$ in $\FCAlg_{\F^\star_\H\InfEoo_{R/k}}$.
\end{remark}

\begin{lemma}\label{lem:higher_graded}
    For each integer $n\geq -1$, there is a natural identification
    $$\gr^*\THH^{(n)}(R/k)\we\Sym^*_R(\L_{R/k}[n])$$ of graded $\bE_\infty$-$R$-algebras.
\end{lemma}

\begin{proof}
    For $n=-1$, this is the content of Lemma~\ref{lem:graded}. The general case follows by
    iteratively applying the Bar construction.
\end{proof}

For the rest of this section, higher $\THH$, the $\bE_\infty$-cotangent complex, and $\bE_\infty$-infinitesimal
cohomology are all computed relative to $\bS$ to simplify the exposition, so we omit the base in the notation.

\begin{example}[Thom spectra]
    Suppose that $A$ is a grouplike $\bE_\infty$-anima and that $f\colon A\rightarrow\mathrm{Pic}(\bS)$ is an
    $\bE_\infty$-map. Let $R=\mathrm{Th}(f)$, the Thom spectrum of $f$, which is an
    $\bE_\infty$-ring.
    A result of Blumberg--Cohen--Schlichtkrull~\cite{blumberg-cohen-schlichtkrull} identifies $\THH(R)$ with
    $R[\B A]\we\Sigma^\infty_+\B A\otimes_\bS R$ as $\bE_\infty$-$R$-algebras;
    see also~\cite{blumberg-thom,krause-nikolaus-lectures}. In fact, at heart, this equivalence
    arises from the Thom isomorphism as studied in~\cite{ando-blumberg-gepner-hopkins-rezk}.
    Indeed, we have $R\otimes_\bS R\we R[A]$. By iteratively applying the Bar
    construction to the augmentation $R[A]\rightarrow R$ arising from the canonical $R$-orientation
    of $R$, we obtain $$\THH^{(n)}(R)\we R[\B^nA]$$ for
    $n\geq 0$. (See~\cite{rasekh-stonek-valenzuela} for related results.) Now, the $\bE_\infty$-cotangent complex of Thom spectra is also understood, by work
    of Basterra--Mandell~\cite{basterra-mandell} (see also~\cite{rasekh-stonek}) which proves that
    $\L_R^{\bE_\infty}\we R\otimes_\bS\B^\infty A$, where $\B^\infty A$ is the connective spectrum
    associated to $A$. The HKR filtration on $\THH^{(n)}(R)$ thus has associated graded pieces
    given by
    $$\gr^*\THH^{(n)}(R)\we\Sym^*_R(R\otimes_\bS\B^\infty A[n])\we R\otimes_\bS\Sym^*(\B^\infty A[n]).$$
    We can reverse the logic and use the computation of $\THH^{(n)}(R)$ to determine
    the cotangent complex $\L_R$ by stabilization. 
\end{example}

Thanks to~\cite[Prop.~3.3.4]{raksit},
there is a symmetric monoidal shearing autoequivalence
$[2*]\colon\GrMod_{\bZ}\rightarrow\GrMod_{\bZ}$. More generally, we have a symmetric monoidal
functor $[2a*]$ obtained by shearing $a$ times.

\begin{lemma}\label{lem:shear}
    Let $R$ be an $\bE_\infty$-$\bZ$-algebra.
    Shearing induces equivalences
    $$[2a*]\colon\gr^*\THH^{(-1)}(R)\we\gr^*\THH^{(-1+a)}(R)\quad\text{and}\quad
    [2a*]\colon\gr^*\THH^{(0)}(R)\we\gr^*\THH^{(a)}(R)$$
    for each integer $a\geq 0$.
\end{lemma}

\begin{proof}
    The claim follows from Lemma~\ref{lem:higher_graded} and shearing.
\end{proof}

Specializing to $R=\bF_p$, 
Lemma~\ref{lem:shear} has an interesting consequence. There are convergent spectral sequences computing $\pi_*\bS_p$
and $\pi_*\THH(\bF_p)$ with the same $\bE^1$-page, up to a shear. Similarly, there are convergent
spectral sequences computing $\pi_*(\bF_p\otimes_\bS\bF_p)$, i.e., the dual Steenrod algebra, and
$\pi_*\THH^{(2)}(\bF_p)$ with the same $\bE_1$-page, up to a shear. Note that
$\THH^{(2)}(\bF_p)\we\bF_p\oplus\bF_p[3]$. In both cases, the first abutment is large and
interesting while the second is rather small.

\begin{remark}
    When $p=2$, one has that $\tfrac{1}{2}$-shearing is also symmetric monoidal by a variant
    of~\cite[Prop.~3.3.4]{raksit}. Thus, in that case,
    the $\bE^1$-pages of the HKR-spectral sequences for $\THH^{(n)}(\bF_2)$ agree (up to shearing)
    for all $n\geq -1$.
\end{remark}

Lemma~\ref{lem:shear} can be used to glean some information about the
structure of the spectral sequences.
We display in Figures~\ref{fig:e1hodgess} and~\ref{fig:e1hkrss}
the $\E^1$-pages for $\F^\star_\H\InfEoo_{\bF_p}$ and $\F^\star_\HKR\THH(\bF_p)$ at $p=2$.
It is easy to see that $\pi_0\L_{\bF_p}^{\bE_\infty}=0$ and that
$\pi_1\L_{\bF_p}^{\bE_\infty}\iso\bF_p$.
Using B\"okstedt's calculation~\cite{bokstedt} that $\pi_*\THH(\bF_p)\iso\bF_p[b]$ where $|b|=2$, it follows that 
the multiplicative generators in the HKR spectral sequence on the slope $2$ diagonal are all
permanent. Everything below that diagonal must die before the $\E^\infty$-page. This implies that
$\pi_i\L_{\bF_p}^{\bE_\infty}=0$ for $i=2,3,4$. Hence, in low degrees,
$\pi_*\Sym^i_{\bF_p}(\L_{\bF_p}^{\bE_\infty}[-1])$ is isomorphic to $\H_*(\B\Sigma_i,\bF_p)$.
This justifies the zeros in the Hodge spectral sequence.

The $0$th column in the Hodge spectral sequence accumulates to $\pi_0\bS_p\iso\bZ_p$. If $p=2$, the $1$
column must have only a single $\bF_2$ surviving on the $\E^\infty$-page. Since
$\H_1(\B\Sigma_2,\bF_2)\iso\bF_2$ cannot support a differential or be hit, it detects the Hopf
map $\eta\in\pi_1\bS_2$. In particular, the Hopf map is detected in filtration weight $2$.

By using that $\THH^{(2)}(\bF_p)$ is $\bF_p\oplus\bF_p[3]$, one can find further that
$\pi_i\L_{\bF_2}^{\bE_\infty}\iso\bF_2$ for $i=5,6,7$.

In fact, the cotangent complex $\L_{\bF_p}^{\bE_\infty}$ is known in terms of Steenrod operations thanks to work of
Lazarev~\cite[Thm.~6.5]{lazarev} who credits unpublished work of Kriz.
Moreover, the homology of the symmetric groups is completely understood in terms of Dyer--Lashof
operations; see for example~\cite[Chap.~VI]{adem-milgram}. It follows that the full $\E^1$-pages of the Hodge and HKR spectral sequences for
$\bF_p$ could be written down. They are rather large.

\begin{sseqdata}[name = EooInfFp, Adams grading, classes = {draw = none } ]
    \foreach \y in {0,...,5} {
        \class["\bullet"](0,\y)
    }

    \foreach \x in {1,2,3} {
        \class["0"](\x,1)
    }

    \foreach \x in {4,5,6} {
        \class["\star"](\x,1)
    }

    \class["\bullet"](1,2)
    \class["\circ"](2,2)

    \foreach \y in {3,4,5} {
        \class["\circ"](1,\y)
    }

    \d1(2,2)

\end{sseqdata}

\begin{figure}[h]
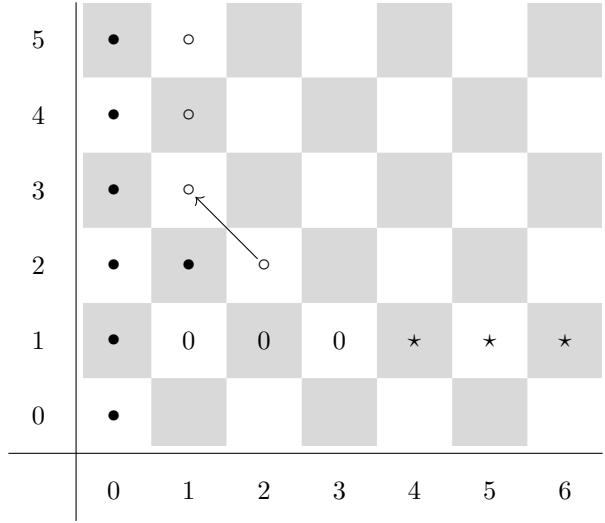

    \centering
    \printpage[ name = EooInfFp, grid=chess, page =
    0,x range = {0}{6},y range = {0}{5}]
    \caption{The $\bE^1$-page of the spectral sequence associated to
    $\F^\star_\H\InfhatEoo_{\bF_2/\bS}$ in Adams grading. The terms labeled $0$ are known to be zero. The bullets
    $\bullet$ are copies of $\bF_2$ which survive to the $\E^\infty$-page. The circles $\circ$ are
    copies of $\bF_2$ which do not survive to the $\E^\infty$-page. The stars $\star$ are copies of
    $\bF_2$ which might or might not support differentials. Undecorated boxes on or above
    the $1$-line are unspecified $\bF_2$-vector spaces.}
    \label{fig:e1hodgess}
\end{figure}

\begin{sseqdata}[name = HKRFp, Adams grading, classes = {draw = none } ]
    \class["\bullet"](0,0)
    \class["\bullet"](2,1)
    \class["\bullet"](4,2)
    \class["\circ"](5,2)
    \class["\bullet"](6,3)
    \class["\circ"](6,1)
    \d1(6,1)
    \class["\circ"](7,1)
    \class["\circ"](6,2)
    \d1(7,1)
    \class["\circ"](7,2)
    \class["\circ"](7,3)
    \class["\circ"](8,1)
    \d1(8,1)
    \foreach \x in {3,4,5} {
        \class["0"](\x,1)
    }
\end{sseqdata}

\begin{figure}[h]
    \centering
    \printpage[ name = HKRFp, grid=chess, page = 0,x range = {0}{7},y range = {0}{3}]
    \caption{The $\bE^1$-page of the spectral sequence associated to
    $\F^\star_\HKR\THH(\bF_2)$ in Adams grading.}
    \label{fig:e1hkrss}
\end{figure}

\begin{remark}
    Pirashvili gives a decomposition of higher $\HH^{(n)}(R/\bQ)$ in the case of
    commutative $\bQ$-algebras and notes the homotopy groups depend (up to a shift) only on the parity of $n$.
    This follows from the decomposition of the HKR filtration in characteristic $0$ (see~\cite{loday})
    followed by iteratively applying Bar constructions to the augmentation
    $\HH(R/\bQ)\rightarrow R$.
    In mixed characteristic, we obtain non-split filtrations with the same
    relationship on associated graded pieces thanks to Lemma~\ref{lem:shear}.
    Related work is carried out by Richter--Ziegenhagen in~\cite{richter-ziegenhagen} in the case
    of $\bE_n$-algebras. We expect this to be related to the material in Section~\ref{sec:en}.
\end{remark}

In the rest of this section we note that instead of the $\bE_\infty$ filtrations on $\THH^{(n)}(R)$
constructed above we can, when $R$ is a commutative ring, use higher $\THH$ variants of the motivic filtration
constructed on $\THH(R;\bZ_p)$ in~\cite{bms2}. We expect that eventually the theory of synthetic
$\bE_\infty$-rings and synthetic $\THH$ being worked out by Devalapurkar--Hahn--Raksit--Senger
should provide a direct bridge between the $\bE_\infty$ and motivic filtrations.

\begin{construction}[Motivic filtration on higher $\THH$]
    Construction~\ref{const:higher_hkr} applies to all maps of $\bE_\infty$-rings. We can instead
    restrict to $\THH(R)=\THH(R/\bS)$ when $R$ is a commutative ring and apply these ideas to the
    motivic filtration constructed by Bhatt--Morrow--Scholze in the $p$-adic case in~\cite{bms2} and in general
    in~\cite{bhatt-lurie-apc,hrw,morin}. Denote this filtration by
    $\F^\star_\mot\THH^{(1)}(R)=\F^\star_\mot\THH(R)$. Taking the Bar construction $(n-1)$ times
    against the augmentation $\F^\star_\mot\THH^{(1)}(R)\rightarrow\gr^0_\mot\THH^{(1)}(R)\we R$ we
    obtain a filtration $\F^\star_\mot\THH^{(n)}(R)$.
\end{construction}

\begin{warning}
    In the $p$-complete case,
    as $\gr^i_\mot\THH(R;\bZ_p)\we\gr^i_\N\Prism_R$, it is tempting to guess that we can compute
    $\F^\star_\mot\THH^{(n)}(R;\bZ_p)$ as in Construction~\ref{const:higher_hkr} by taking factorization
    homology internal to the $\infty$-category of filtered $p$-complete $\N^{\geq\star}\Prism_R$-modules of the
    map $\N^{\geq\star}\Prism_R\rightarrow\gr^0_\N\Prism_R\we R$. However, if this were the case,
    then $\THH^{(n)}(R)$ would in particular admit the structure of a connective derived commutative ring. The
    case of $R=\bF_p$ and $n=1$ shows that this is not the case. By B\"okstedt~\cite{bokstedt} we
    have that $\pi_*\THH(\bF_p)\iso\bF_p[b]$, a polynomial ring on a degree $2$ generator. But,
    degree $2$ homotopy classes of connective derived commutative admit divided powers, so this is
    impossible.
\end{warning}

\begin{example}\label{ex:thh2_fp}
    We have $\F^\star_\mot\THH^{(1)}(\bF_p)\we\tau_{\geq 2\star}\THH^{(1)}(\bF_p)$. Thus,
    $\F^\star_\mot\THH^{(2)}(\bF_p)$ is given by the two-step multiplicative filtration on
    $\THH^{(2)}(\bF_p)$ with $\gr^1\we\bF_p[3]$ and $\gr^0\we\bF_p$. Then, we have that
    $\pi_*\THH^{(3)}(\bF_p)$ is a free divided power algebra on a degree $4$ class $\sigma^2 b$ and
    $\F^\star_\mot\THH^{(3)}(\bF_p)\we\tau_{\geq 4\star}\THH^{(3)}(\bF_p)$.
\end{example}

\begin{remark}
    The motivic filtration on $\THH^{(1)}(\bF_p)$ agrees with the even filtration of~\cite{hrw}. This is the case
    for all quasisyntomic rings. By its construction, $\gr^i$ of the even filtration is
    $2i$-coconnective. Thus, the motivic filtration on $\THH^{(2)}(\bF_p)$ examined in
    Example~\ref{ex:thh2_fp} is not the even filtration.
\end{remark}

\begin{remark}
    As in~\cite[Sec.~11]{bms2},
    we can make the motivic filtration work for higher $\THH$ over certain other base
    $\bE_\infty$-rings. For example, consider $\bS[z]\rightarrow\bZ_p$ where $z\mapsto p$.
    Then, $\pi_*\THH(\bZ_p/\bS[z];\bZ_p)\iso\bZ_p[b]$, where $|b|=2$. The motivic filtration here is the
    double speed Postnikov filtration and we can Bar this to obtain motivic filtrations on
    $\THH^{(n)}(\bZ_p/\bS[z];\bZ_p)$, the $p$-completion of $\THH^{(n)}(\bZ_p/\bS[z])$. This
    technique can further be combined with the descent method
    of Krause--Nikolaus~\cite{krause-nikolaus} to obtain new computations of $\THH^{(n)}(\Oscr_K;\bZ_p)$ for rings
    of integers $\Oscr_K$ in $p$-adic fields $K$. For example, Krause and Nikolaus compute
    $\pi_*\THH^{(1)}(\bZ_p/\bS[z];\bZ_p)$ as the homology of a dga
    $\bZ_p[b]\otimes_{\bZ_p}\Lambda_{\bZ_p}(y)$
    where $|b|=2$, $|y|=1$, and $\d(b)=y$. Thus, $\d(b^n)=nb^{n-1}y$ which gives that
    $\pi_{2n-1}\THH^{(1)}(\bZ_p/\bS[z];\bZ_p)\iso\bZ_p/n$ for $n\geq 1$ with the rest of the positive
    groups vanishing. By taking the Bar construction of this dga we obtain $\Lambda_{\bZ_p}(\sigma
    b)\otimes_{\bZ_p}\Gamma_{\bZ_p}(\sigma y)$ with $\d(\sigma b)=\sigma y$ which computes
    $\pi_*\THH^{(2)}(\bZ_p;\bZ_p)$. Now, we have $\d(\sigma b (\sigma y)^{[n]})=(\sigma y)(\sigma
    y)^{[n]}=(n+1)(\sigma y)^{[n+1]}$. Thus, we recover the calculation
    of~\cite{dundas-lindenstrauss-richter} that $\pi_*\THH(\bZ_p;\bZ_p)$ is
    $\Gamma_{\bZ_p}(z)/(z)$, the free divided power algebra on a degree $2$ element modulo the
    ideal generated by that element.
\end{remark}

\section{An $\bE_\infty$-version of de Rham cohomology}

In the case of derived commutative rings, the existence of de Rham cohomology is intimately linked
to the notion of shearing. Over $\bS$, shearing is not symmetric monoidal. But, it can be used to
equip $\GrMod_\bS$ with an exotic symmetric monoidal structure. Let $k$ be an $\bE_\infty$-ring.
Let $$[-2*]\colon\GrMod_k\rightarrow\GrMod_k$$ be the shearing (downward) equivalence. Using it, we can
transport the usual symmetric monoidal structure $\GrMod_k^{\otimes_\D}$ on $\GrMod_k$ to a new
symmetric monoidal structure $\GrMod_k^{\otimes_\S}$ on $\GrMod_k$. 
This structure is uniquely determined by the requirement that shearing admit an enhancement to a
symmetric monoidal equivalence $$[-2*]\colon\GrMod^{\otimes_\D}\rightarrow\GrMod^{\otimes_\S}.$$
Note that $[-2*]$ is in fact $\bE_2$-monoidal, so the underlying $\bE_2$-monoidal structure of
$\GrMod^{\otimes_\S}$ is equivalent to that of $\GrMod^{\otimes_\D}$.

Now, consider $\bT_\gr^\vee\we k(0)\oplus k(-1)[-1]$, the trivial square-zero extension of $k(0)$ by
$k(-1)[-1]$ as an $\bE_\infty$-ring in $\GrMod_k^{\otimes_\D}$. As with any trivial square-zero
extension, $\bT_\gr^\vee$ admits the structure of a bicommutative bialgebra, i.e., it is naturally a
cocommutative coalgebra in $\CAlg(\GrMod_k^{\otimes_\D})$. (This graded circle exists whether not a
filtered circle exists.) We can shear down to obtain
$\bD_-^\vee=\bT_\gr^\vee[-2*]$ in $\CAlg(\GrMod_k^{\otimes_\S})$.

\begin{definition}
    Let $\widehat{\FMod}_k^{\otimes_\S}=\cMod_{\bD_-^\vee}(\GrMod_k^{\otimes_\S})$.
    The underlying $\infty$-category of $\widehat{\FMod}_k^{\otimes_\S}$ is $\widehat{\FMod}_k$.
    The symmetric monoidal structure is a twist of the usual symmetric monoidal structure on
    complete filtered $k$-modules in the sense that it agrees up to $\bE_2$.
\end{definition}

\begin{lemma}
    There is a symmetric monoidal structure $\FMod_k^{\otimes_\S}$ on $\FMod_\S$ which is
    determined by the criteria that the completion functor
    $\FMod_k^{\otimes_\S}\rightarrow\widehat{\FMod}_k^{\otimes_\S}$ be symmetric monoidal and the inclusion
    functor $\FMod_k^{\otimes_\S}\leftarrow\widehat{\FMod}_k^{\otimes_\S}$ is symmetric monoidal on
    compact objects of $\FMod_k$ (which are all complete). 
\end{lemma}

\begin{proof}
    The compact objects $\FMod_k^\omega\subseteq\FMod_k$ are all complete and are closed under
    $\otimes_\S$ inside $\FMod_k^\wedge$. Thus, we obtain a symmetric monoidal structure on
    $\Ind(\FMod_k^\omega)\we\FMod_k$. We leave it to the reader to check that it has the desired
    properties.
\end{proof}

The following definition was suggested to us by Arpon Raksit.

\begin{definition}[$\bE_\infty$-de Rham cohomology]
    We let $$\F^\star\dR^{\bE_\infty}_{-/k}\colon\CAlg_k\rightarrow\CAlg(\FMod_k^{\otimes_\S})$$ be
    the left adjoint of the $\gr^0$ functor. This is Hodge-filtered de Rham cohomology of
    $\bE_\infty$-$k$-algebras. Similarly, let
    $\F^\star\dRhat^{\bE_\infty}_{-/k}\colon\CAlg_k\rightarrow\CAlg(\widehat{\FMod}_k^{\otimes_\S})$
    be the left adjoint of the $\gr^0$ functor. It defines Hodge-complete Hodge-filtered de Rham
    cohomology of $\bE_\infty$-rings.
\end{definition}

\begin{remark}
    When a filtered circle exists over $k$ in the sense that there is a two step filtration on
    $k[S^1]$ with graded pieces $k$ in weight $0$ and $k$ in weight $1$ and where moreover this is
    an $\bE_\infty$-algebra object in coalgebras over $k$, then one can construct the HKR
    filtration as in~\cite{raksit}, take associated graded pieces, and shear down to recover
    $\bE_\infty$-de Rham cohomology.
\end{remark}

\begin{remark}
    For an $\bE_\infty$-$R$-algebra $R$,
    we have $\gr^*_\H\dR^{\bE_\infty}_{R/k}\we\Sym^*_R(\L_{R/k}^{\bE_\infty}[1])[-2*]$.
    The descent results for Hodge-completed $\bE_\infty$-infinitesimal cohomology are available for
    Hodge-completed $\bE_\infty$-de Rham cohomology. If $k$ and $R$ are connective, then we can use descent
    along $k\rightarrow\pi_0k$ to get into a situation where shearing is symmetric monoidal and
    hence the $\bE_\infty$-infinitesimal and $\bE_\infty$-de Rham theories agree. This descent
    strategy can be used to establish properties of $\bE_\infty$-de Rham cohomology. However, we have
    no particular applications for these at the moment.
\end{remark}

\section{$\bE_n$-infinitesimal cohomology and HKR}\label{sec:en}

There is a natural variant of $\bE_\infty$-infinitesimal cohomology.

\begin{definition}\label{def:en}
    Fix an $\bE_\infty$-ring $k$ and an integer $n\geq 0$.
    Let $\F\Alg^{\bE_n}_k$ be the $\infty$-category of filtered $\bE_n$-algebras over $k$, i.e.,
    the $\infty$-category of $\bE_n$-algebras in $\FMod_k$. This admits a limit and colimit-preserving
    functor $\Alg^{\bE_n}_k\leftarrow\F\Alg_k^{\bE_n}\colon\gr^0$. The left adjoint
    $$\F^\star_\H\InfEn_{-/k}\colon\Alg^{\bE_n}_k\rightarrow\F\Alg_k^{\bE_n}$$
    is called the $\bE_n$-infinitesimal cohomology functor.
    Let $\F^\star_\H\InfhatEn_{-/k}$ be its completion.
\end{definition}

An $\bE_n$-algebra $R$ has an $\bE_n$-monoidal $\infty$-category $\Mod_R^{\bE_n}$ of $\bE_n$-modules. For
example, when $n=1$, this is the $\infty$-category of $(R,R)$-bimodules. When $n=\infty$,
$\Mod_R^{\bE_\infty}\we\LMod_R$. We also have $\GrMod_R^{\bE_n}$, an $\bE_n$-monoidal category of
graded $\bE_n$-modules over $R$.

\begin{lemma}
    Given $R\in\Alg^{\bE_n}_k$, there is a natural equivalence between $\gr^*_\H\InfEn_{R/k}$ and the
    free $\bE_n$-algebra in $\GrMod^{\bE_n}_R$ on $\L_{R/k}^{\bE_n}[-1](1)$.
\end{lemma}

\begin{proof}
    The proof is the same as in the $\bE_\infty$-case up to the following observations. Firstly,
    $\Alg^{\bE_n}(\GrMod_R^{\bE_n})\we(\Alg^{\bE_n}\GrMod_k))_{R(0)/}$.
    Restricting to graded pieces $0$ and $1$, we find that $\F^{[0,1]}_\H\InfEn_{R/k}$ is the
    universal $\bE_n$-square-zero extension of $R$ and hence that
    $\gr^1_\H\InfEn_{R/k}\we\L_{R/k}^{\bE_n}[-1]$. This gives a natural map from the free
    $\bE_n$-algebra in $\GrMod^{\bE_n}$ on $\L_{R/k}^{\bE_n}[-1](1)$ to $\gr^*_\H\InfEn_{R/k}$.
    To prove that it's an equivalence, it is enough to prove it on free $\bE_n$-algebras. Let
    $R=\Fr^{\bE_n}_k(X)$ be the free $\bE_n$-algebra on an $k$-module $X$. Then,
    $\gr^*_\H\InfEn_{R/k}$ is equivalent to the free graded $\bE_n$-algebra on $X(0)\oplus
    X[-1](1)$. This in turn is the free graded $\bE_n$-algebra under $R(0)$ on the free
    $\bE_n$-module over $R$ on $X[-1](1)$. This free $\bE_n$-module over $R$ on $X[-1](1)$ is
    precisely $\L_{R/k}^{\bE_n}[-1](1)$ (see~\cite[Lem.~2.12]{francis}). Thus, both sides have the same form and the induced map is
    an equivalence in weights $0$ and $1$. Since the algebras are free, it is an equivalence.
\end{proof}

\begin{remark}[The $\bE_n$-cotangent complex]
    Given an $\bE_n$-algebra $R$, by~\cite[Thm.~2.26]{francis} or~\cite[Thm.~7.3.5.1]{ha}, there is a canonical fiber sequence $$\int_{S^{n-1}}R\rightarrow
    R\rightarrow\L_{R/k}^{\bE_n}[n]$$ in $\Mod_R^{\bE_n}$.
\end{remark}

\begin{example}
    Since $\bF_p$ is $\bE_\infty$, we can compute $\int_{S^{2}}\bF_p$ as
    $\bF_p\otimes_{\THH(\bF_p)}\bF_p$. By B\"okstedt's calculation~\cite{bokstedt}, we have
    $\pi_*\THH(\bF_p)\iso\bF_p[b]$, the polynomial ring on a degree $2$ generator. Thus,
    $\int_{S^2}\bF_p\we\bF_p\oplus\bF_p[3]$. It follows that $\L_{\bF_p}^{\bE_3}\we\bF_p[1]$.
\end{example}

\begin{example}[$n=1$]\label{ex:e1}
    Suppose that $k$ is an $\bE_\infty$-ring and $R$ is an $\bE_1$-$k$-algebra, i.e., and
    $\bE_1$-algebra in $\LMod_k$. There is a cobar construction $R^\bullet$ which is a cosimplicial
    left $k$-module of the form $$R\stack{3}R\otimes_kR\stack{5}R\otimes_kR\otimes_kR\cdots.$$
    If we let $I=\fib(k\rightarrow R)$ be the ideal associated to $R$, then by~\cite[Prop.~2.14]{MNN17} there are natural fiber
    sequences $I^{\otimes (n+1)}\rightarrow k\rightarrow\Tot^n(R^\bullet)$.
    It follows from the monoidal version of the (cosimplicial, augmented) Dold--Kan correspondence presented
    in~\cite[Prop.~1.5]{gammage-hilburn-mazelgee}
    that the induced tower $$\cdots\rightarrow I^{\otimes 3}\rightarrow I^{\otimes
    2}\rightarrow I\rightarrow k$$ is equivalent to $\F^\star_\H\Inf^{\bE_1}_{R/k}$. This example was also
    noted in~\cite{lawson-video}.
\end{example}

\begin{remark}
    Burklund~\cite{burklund-moore} has constructed $\bE_n$-algebra structures on $\bS/p^r$ for
    suitable $r$. When $n=1$ we see from Example~\ref{ex:e1} that $\bS/p^r$ admits an
    $\bE_1$-algebra structure if and only the $p^r$-adic filtration on $\bS$ admits an
    $\bE_1$-structure.
\end{remark}

We are indebted to Tyler Lawson for the details in the following example.

\begin{example}[Comparison to Hill--Lawson]\label{ex:hill_lawson}
    By the previous example,
    the $\bE_3$-cotangent complex of $\bF_p$ is $\bF_p[1]$ as an $\bE_3$-module over $\bF_p$.
    However, this $\bE_3$-module is central in the sense that it is in the image of the
    $\bE_3$-monoidal forgetful functor $\Mod_{\bF_p}^{\bE_3}\leftarrow\LMod_{\bF_p}$ (produced
    implicitly in~\cite[Rem.~5.1.3.3]{ha}). This forgetful functor
    also preserves colimits. It follows that
    $\gr^*_\H\Inf^{\bE_3}_{\bF_p}$ is equivalent, as an $\bE_3$-algebra with a map from $\bF_p$, to
    the free graded $\bE_3$-algebra $\Fr^{\bE_3}_{\bF_p}(\bF_p)$ on $\bF_p(1)$ in the symmetric monoidal $\infty$-category
    $\Gr\LMod_{\bF_p}$.

    In~\cite{hill-lawson}, Hill and Lawson construct an $\bE_2$-filtration on the sphere spectrum whose
    associated graded is also equivalent to the free graded $\bE_3$-algebra
    $\Fr^{\bE_3}_{\bF_p}(\bF_p(1))$. They use the fact that, as an $\bE_2$-ring, $\bF_p$ is a
    pushout
    $$\xymatrix{
        \Fr^{\bE_2}(x)\ar[r]^{x\mapsto p}\ar[d]_{x\mapsto 0}&\bS\ar[d]\\
        \bS\ar[r]&\bF_p
    }$$
    by Hopkins--Mahowald, where $x$ denotes a degree $0$ generator. Hill and Lawson show that this implies that $\bF_p$ is also realized as
    a Bar construction $\bS\otimes_{\Fr^{\bE_3}(x)}\bS$ over the free $\bE_3$-algebra on $x$.
    They define their filtration by letting
    $$\F^\star_{\mathrm{HL}}\bS=\Fr^{\bE_3}(x_1)\otimes_{\Fr^{\bE_3}(x_0)}\bS,$$
    where $x_0$ maps to $\tau x_1$ on the left and to $p$ on the right. A priori, this tensor
    product is an $\bE_2$-ring. Nevertheless, since $p\in\bS$ lifts to
    $p\in\F^1_\H\Inf^{\bE_{\infty}}_{\bF_p}$, there is a commutative diagram
    $$\xymatrix{
        \Fr^{\bE_3}(x_0)\ar[r]\ar[d]&\ins^0\bS\ar[d]\\
        \Fr^{\bE_3}(x_1)\ar[r]&\F^\star_\H\Inf^{\bE_3}_{\bF_p}
    }$$
    of filtered $\bE_3$-algebras. This induces an $\bE_2$-monoidal map
    $\F^\star_{\mathrm{HL}}\bS\rightarrow\F^\star_\H\Inf^{\bE_3}_{\bF_p}$, for example
    by~\cite[Lem.~2.3]{heuts-land}.
    We claim this map is an equivalence. It is an equivalence upon applying $\F^0$ and both sides have equivalent associated graded rings, so it
    is enough to show that the induced map on $\gr^1$ is an equivalence or even non-zero.
    But, $p$ is detected in both $\gr^1_{\mathrm{HL}}\bS$ and $\gr^1\Inf_{\bF_p}^{\bE_3}$
    for degree reasons, which completes the proof.
\end{example}

\begin{remark}
    We do not know if $\bS_p\we\Infhat^{\bE_3}_{\bF_p}$, although $\bS_p$ is a retract of
    $\Infhat^{\bE_3}_{\bF_p}$ thanks to the natural map
    $\Infhat^{\bE_3}_{\bF_p}\rightarrow\Infhat^{\bE_\infty}_{\bF_p}\we\bS_p$. To study this question in detail would require some results
    for descent for $\bE_n$-cotangent complexes which is outside the scope of this article.
\end{remark}

We now discuss a version of the HKR filtration for $\THH$ of $\bE_n$-rings.

\begin{definition}[$\THH$ relative to infinitesimal cohomology]
    If $k$ is an $\bE_\infty$-ring and $R\in\Alg_{\bE_n}(\LMod_k)$, then
    $\F^\star_\H\Inf^{\bE_n}_{R/k}\rightarrow R$ can be viewed as a map of $\bE_n$-algebras in
    $\FMod_k$. It follows that we can view $R$ as an $\bE_n$-algebra in the $\bE_n$-monoidal category
    $\FMod_{\F^\star_\H\Inf^{\bE_n}_{R/k}}^{\bE_n}$. If $n\geq 2$, then we let
    $\F^\star_{\HKR}\THH^{\bE_n}(R/k)=\THH(R/\F^\star_\H\Inf^{\bE_n}_{R/k})$, computed in the $\bE_n$-monoidal
    category $\FMod_{\F^\star_\H\Inf^{\bE_n}_{R/k}}^{\bE_n}$. The resulting object
    $\F^\star_{\HKR}\THH^{\bE_n}(R/k)$ is an $\bE_{n-2}$-algebra in $\FMod_k$.
    That $\THH$ makes sense in the
    context of $\bE_2$-monoidal categories can be found in~\cite[Rem.~3.15]{krause-nikolaus}.
\end{definition}

\begin{example}
    We have that $\F^\star_{\HKR}\THH^{\bE_3}(\bF_p)\we\tau_{\geq 2\star}\THH(\bF_p)$ as
    $\bE_1$-rings. Indeed, $\F^\star_{\HKR}\THH^{\bE_3}(\bF_p)$ is obtained by applying the Bar construction
    twice to the augmentation $\F^\star_\H\Inf^{\bE_3}_{\bF_p}\rightarrow\bF_p$. 
    Thus, on associated graded pieces, the result is the free graded $\bE_1$-algebra under $\bF_p$ on
    $\L_{\bF_p}^{\bE_3}[1](1)\we\bF_p[2](1)$. As in Example~\ref{ex:hill_lawson}, since $\bF_p[2]$ is central, this is in fact
    equivalent to the free graded $\bE_1$-algebra in $\LMod_{\bF_p}$ on $\bF_p[2](1)$.
    We can Bar once more along the $\bE_1$-map $\THH^{\bE_3}_\fil(\bF_p)\rightarrow\bF_p$ to obtain
    an $\bE_0$-filtration $\F^\star_{\HKR}\THH^{(2),\bE_3}(\bF_p)$ on $\THH^{(2)}(\bF_p)$. On associated graded pieces
    it is precisely the free graded $\bE_0$-algebra in $\LMod_{\bF_p}$ on
    $\L_{\bF_p}^{\bE_3}[2](1)\we\bF_p[3](1)$.
\end{example}

\begin{remark}
    These constructions should have an interpretation in terms of factorization internal to
    the $\bE_n$-monoidal category $\FMod_{\F^\star_\H\Inf^{\bE_n}_{R/k}}^{\bE_n}$ as in
    Construction~\ref{const:higher_hkr}.
    It should be related to the $\bE_n$-homology theories studied in Fresse~\cite{fresse} and by
    Livernet--Richter~\cite{livernet-richter}.
    Such a theory is not in the literature at the moment and we do not explore it further.
\end{remark}

\section{A connection to Goodwillie calculus}\label{sec:goodwillie}

We thank John Francis for the comment that led to this section.
Filtrations on augmented $\Oscr$-algebras, Koszul duality, and Goodwillie calculus have a long,
intertwined history. See for
instance~\cite{ayala-francis-duality,ching-harper,francis-gaitsgory,gaitsgory-rozenblyum-2,harper-hess,kuhn_mccord,kuhn-pereira,richter-ziegenhagen}.
We will follow the survey in~\cite{brantner-mathew}.

For simplicity, we work in the setting of $\bE_\infty$-rings, but much works in much greater
generality. Let $k$ be a fixed $\bE_\infty$-ring and consider $\CAlg_k^\aug$, the $\infty$-category
of augmented $k$-algebras. Brantner and Mathew explain in~\cite{brantner-mathew} how to use the
comonad associated to the
adjunction $\mathrm{cot}\colon\CAlg_k^\aug\rightleftarrows\Mod_k\colon\mathrm{sqz}$ to develop a theory
of Lie algebras which controls deformation theories. Here, $\mathrm{cot}(R)$ is the cotangent fiber
$k\otimes_R\L_{R/k}$ while $\mathrm{sqz}$ is the square-zero extension functor.

To do this, and control finiteness properties, they introduce a filtered analogue. Let $\F^{\geq
1}\Mod_k=\Fun(\bZ_{\geq 1}^\op,\Mod_k)$ be the $\infty$-category of $\bZ_{\geq 1}$-indexed
decreasing filtrations. This is a non-unital symmetric monoidal category. One can consider it as a
full non-unital symmetric monoidal subcategory of $\F^+\Mod_k$ on those filtrations $\F^\star M$ where $\gr^0M\we 0$.
Let $\CAlg^{\mathrm{nu}}(\F^{\geq 1}\Mod_k)$ be the $\infty$-category of non-unital
$\bE_\infty$-rings in $\F^{\geq 1}\Mod_k$. Equivalently, this is the $\infty$-category of augmented
$\bE_\infty$-algebras $\F^\star R$ in $\F^+\Mod_k$ such that $k\we\gr^0R$.
Note that for such a filtered $\bE_\infty$-algebra, the augmentation is not extra data as there is
always a canonical map $\F^\star R\rightarrow\gr^0R$. Thus, we will write $\F^+\CAlg_k^\red$ for
this instantiation of $\CAlg^{\mathrm{nu}}(\F^{\geq 1}\Mod_k)$.
They consider a commutative diagram
$$\xymatrix{
    \Mod_k\ar[r]^{\Sym^{\mathrm{nu}}}\ar[d]^{\ins^1}&\CAlg_k^{\aug}\ar[r]^{\cot}\ar[d]^{\mathrm{adic}}&\Mod_k\ar[d]^{\ins^1}\\
    \F^{\geq 1}\Mod_k\ar[r]^{\Sym^{\mathrm{nu}}}&\F^+\CAlg_k^{\red}\ar[r]^{\cot}&\F^{\geq 1}\Mod_k
}$$
of left adjoint functors.

\begin{lemma}
    Given an augmented $\bE_\infty$-$k$-algebra $R$, $\mathrm{adic}(R)$ is equivalent to
    $\F^\star_\H\InfEoo_{k/R}$.
\end{lemma}

\begin{proof}
    The right adjoint of $\mathrm{adic}$ is the functor that takes a reduced $\bE_\infty$-ring
    $\F^\star S$ in
    $\F^+\Mod_k$ to the fiber of $\F^0S\rightarrow\gr^0S$. In other words it is the functor
    $\F^\star S\mapsto\F^1S$. Using that in general $\bE_\infty$-infinitesimal cohomology is the
    left adjoint of the $\gr^0$-functor, the claim is now a diagram chase.
\end{proof}

\begin{remark}
    The $\bE_\infty$-infinitesimal cohomology $\InfEoo_{k/R}$ is the cohomological analogue of the
    cotangent fiber. In fact, in the $\bE_\infty$-version of the theory of derived foliations as
    developed by To\"en--Vezzosi~\cite{tv-inf},
    $\F^\star_\H\InfEoo_{k/R}$ corresponds to the leaf of the final foliation on $\Spec R$ that
    passes through the point $x\in\Spec R$ corresponding to the augmentation.
    This connection is made much more explicit in work of
    Branter--Magidson--Nuiten~\cite{brantner-magidson-nuiten} and
    Fu~\cite{fu-inf}.
\end{remark}

\begin{lemma}\label{lem:goodwillie}
    The adic filtration is the Goodwillie tower. In particular, $\F^\star_\H\InfEoo_{k/R}$ is the
    Goodwillie tower of the forgetful functor $\CAlg_k^{\aug}\rightarrow\Mod_k$.
\end{lemma}

\begin{proof}
    See~\cite[Thm.~2.1.6]{ayala-francis-duality} or~\cite[Rem.~1.14]{harper-hess}.
\end{proof}

\begin{remark}\label{rem:minasian2}
    It follows that our construction of $\F^\star_\HKR\THH(R)$ agrees with that of
    Minasian~\cite{minasian}. Indeed, by the claim on p.278 of~\cite{minasian}, Minasian's construction is precisely the adic filtration on the
    augmented $R$-algebra $R\rightarrow\THH(R)\rightarrow R$. We see that this is
    $\F^\star_\H\InfEoo_{R/\THH(R)}$.
\end{remark}

\begin{remark}[Comparison to Glasman]
    One final HKR-type filtration on $\THH$ we are aware of in the literature is due to Saul
    Glasman~\cite{glasman-hodge}. His filtration $\G^\star\THH(R)$ on $\THH(R)$ is multiplicative with
    $\gr^0_\G\THH(R)\we R$. It follows that there is a natural map
    $\F^\star_\HKR\THH(R)\we\F^\star_\H\InfEoo_{R/\THH(R)}\rightarrow\G^\star\THH(R)$.
    We do not at present compare these constructions. However, it should be related to
    the Koszul duality between the numerical (point counting) and Goodwillie--Weiss filtrations as
    considered in~\cite{ayala-francis-duality}.
\end{remark}

\small
\bibliographystyle{amsplain}
\bibliography{eooinf}
\addcontentsline{toc}{section}{References}

\medskip
\noindent
\textsc{Department of Mathematics, Northwestern University}\\
{\ttfamily antieau@northwestern.edu}

\end{document}

%% file: preamble.tex
\usepackage[pdfstartview=FitH,
            pdfauthor={},
            pdftitle={},
            colorlinks,
            linkcolor=reference,
            citecolor=citation,
            urlcolor=e-mail
            ]{hyperref}

\usepackage{amsmath}
\usepackage{amscd}
\usepackage{amsbsy}
\usepackage{amssymb}
\usepackage{verbatim}
\usepackage{eufrak}
\usepackage{eucal}
\usepackage{microtype}
\usepackage{hyperref}
\usepackage{mathrsfs}
\usepackage{amsthm}
\usepackage{stmaryrd}
\usepackage{wasysym}
\usepackage{xy}
\input xy
\xyoption{all}
\usepackage{tikz}
\usetikzlibrary{matrix,arrows}
\usepackage{tikz-cd}
\usepackage{dsfont}
\usepackage[T1]{fontenc}
\usepackage{spectralsequences}

\usepackage{fancyhdr}
\usepackage{makeidx}
\makeindex

\newcommand{\stackspace}{2.5}
\newcommand{\stack}[2][1cm]{\;\tikz[baseline, yshift=.65ex]%
    {\foreach \k [evaluate=\k as \r using (.5*#2+.5-\k)*\stackspace] in {1,...,#2}{%
    \ifodd\k{\draw[->](0,\r pt)--(#1,\r pt);}%
    \else{\draw[<-](0,\r pt)--(#1,\r pt);}\fi
    }}\;}

\usepackage[bbgreekl]{mathbbol}
\usepackage{amsfonts}
\DeclareSymbolFontAlphabet{\mathbb}{AMSb} % to ensure \mathbb does not change
\DeclareSymbolFontAlphabet{\mathbbl}{bbold}
\newcommand{\Prism}{{\mathbbl{\Delta}}}

\newcommand{\Inf}{{\mathbbl{\Pi}}}
\newcommand{\Infhat}{\widehat{\Inf}}

\newcommand{\InfEoo}{\Inf^{\bE_\infty}}
\newcommand{\InfhatEoo}{\Infhat^{\bE_\infty}}
\newcommand{\InfEn}{\Inf^{\bE_n}}
\newcommand{\InfhatEn}{\Infhat^{\bE_n}}

\usepackage{color}
\definecolor{todo}{rgb}{1,0,0}
\definecolor{conditional}{rgb}{0,1,0}
\definecolor{e-mail}{rgb}{0,.40,.80}
\definecolor{reference}{rgb}{.20,.60,.22}
\definecolor{mrnumber}{rgb}{.80,.40,0}
\definecolor{citation}{rgb}{0,.40,.80}

\let\oldmarginpar\marginpar
\renewcommand\marginpar[1]{\-\oldmarginpar[\raggedleft\footnotesize #1]%
{\raggedright\footnotesize #1}}

\newcommand{\Oscr}{\mathcal{O}}

\newcommand{\B}{\mathrm{B}}

\renewcommand{\d}{\mathrm{d}}
\newcommand{\D}{\mathrm{D}}
\newcommand{\E}{\mathrm{E}}
\newcommand{\F}{\mathrm{F}}
\newcommand{\G}{\mathrm{G}}
\renewcommand{\H}{\mathrm{H}}

\renewcommand{\L}{\mathrm{L}}

\newcommand{\N}{\mathrm{N}}

\renewcommand{\r}{\mathrm{r}}
\renewcommand{\S}{\mathrm{S}}

\newcommand{\bD}{\mathbf{D}}
\newcommand{\bE}{\mathbf{E}}
\newcommand{\bF}{\mathbf{F}}

\newcommand{\bQ}{\mathbf{Q}}

\newcommand{\bS}{\mathbf{S}}
\newcommand{\bT}{\mathbf{T}}

\newcommand{\bZ}{\mathbf{Z}}

\newcommand{\op}{\mathrm{op}}

\newcommand{\fib}{\mathrm{fib}}

\newcommand{\Mod}{\mathrm{Mod}}
\newcommand{\LMod}{\mathrm{LMod}}
\newcommand{\FMod}{\mathrm{FMod}}

\newcommand{\GrMod}{\mathrm{GrMod}}

\newcommand{\Ind}{\mathrm{Ind}}

\newcommand{\Alg}{\mathrm{Alg}}
\newcommand{\CAlg}{\mathrm{CAlg}}

\newcommand{\cMod}{\mathrm{cMod}}

\newcommand{\FCAlg}{\mathrm{FCAlg}}

\newcommand{\Gr}{\mathrm{Gr}}
\newcommand{\gr}{\mathrm{gr}}

\newcommand{\ins}{\mathrm{ins}}
\newcommand{\fil}{\mathrm{fil}}

\newcommand{\Sym}{\mathrm{Sym}}

\newcommand{\mot}{\mathrm{mot}}
\newcommand{\red}{\mathrm{red}}
\newcommand{\aug}{\mathrm{aug}}
\newcommand{\Fr}{\mathrm{Fr}}
\newcommand{\Pairs}{\mathrm{Pairs}}

\newcommand{\Dec}{\mathrm{Dec}}
\newcommand{\HKR}{\mathrm{HKR}}

\renewcommand{\geq}{\geqslant}

\newcommand{\THH}{\mathrm{THH}}

\newcommand{\HH}{\mathrm{HH}}

\newcommand{\dR}{\mathrm{dR}}
\newcommand{\dRhat}{\widehat{\dR}}

\newcommand{\GM}{\mathrm{GM}}

\newcommand{\Fun}{\mathrm{Fun}}

\DeclareMathOperator*{\colim}{colim}

\DeclareMathOperator*{\Tot}{Tot}

\DeclareMathOperator{\Spec}{Spec}

\newcommand{\we}{\simeq}
\newcommand{\iso}{\cong}

\theoremstyle{plain}
\newtheorem{theorem}{Theorem}[section]
\newtheorem*{theorem*}{Theorem}
\newtheorem{lemma}[theorem]{Lemma}

\newtheorem{proposition}[theorem]{Proposition}

\newtheorem{corollary}[theorem]{Corollary}
\newtheorem*{corollary*}{Corollary}

\theoremstyle{plain}

\theoremstyle{definition}

\newtheoremstyle{named}{}{}{\itshape}{}{\bfseries}{.}{.5em}{#1 \thmnote{#3}}
\theoremstyle{named}

\theoremstyle{definition}

\newtheorem{definition}[theorem]{Definition}
\newtheorem{warning}[theorem]{Warning}

\newtheorem{example}[theorem]{Example}
\newtheorem*{example*}{Example}

\newtheorem*{question*}{Question}
\newtheorem{construction}[theorem]{Construction}

\newtheorem{remark}[theorem]{Remark}

\newtheorem{motto}[theorem]{Motto}

\AtBeginDocument{%
   \def\MR#1{}
}